\documentclass{article}
\usepackage{caption}
\usepackage{subfigure}
\usepackage{makeidx}
\usepackage{epsfig}
\usepackage{graphics}
\usepackage{latexsym}
\usepackage{amsmath}
\usepackage{amsfonts}
\usepackage{amssymb}
\usepackage[mathscr]{eucal}
\usepackage{mathrsfs}
\usepackage{pifont}
\usepackage{yhmath}
\usepackage{undertilde}
\usepackage[usenames]{color}
\usepackage{bbm}
\usepackage{caption}
\usepackage{booktabs}

\usepackage{url}

\usepackage{fullpage}
\usepackage{epsfig}
\usepackage{graphics}
\usepackage{latexsym}
\usepackage{amsmath}
\usepackage{amsfonts}
\usepackage{amssymb}
\usepackage{mathrsfs}
\usepackage{pifont}
\usepackage{yhmath}
\usepackage{undertilde}
\usepackage{bbm}
\usepackage{dsfont}
\usepackage{eufrak}
\usepackage{amsthm}

\newtheorem{theorem}{Theorem}[section]
\newtheorem{lemma}[theorem]{Lemma}

\newtheorem{definition}[theorem]{Definition}

\theoremstyle{remark}

\newtheorem{claim}{Claim}

\begin{document}

\newcounter{my}
\newenvironment{mylabel}
{
\begin{list}{(\roman{my})}{
\setlength{\parsep}{-1mm}
\setlength{\labelwidth}{8mm}
\usecounter{my}}
}{\end{list}}

\newcounter{my2}
\newenvironment{mylabel2}
{
\begin{list}{(\alph{my2})}{
\setlength{\parsep}{-0mm} \setlength{\labelwidth}{8mm}
\setlength{\leftmargin}{3mm}
\usecounter{my2}}
}{\end{list}}

\newcounter{my3}
\newenvironment{mylabel3}
{
\begin{list}{(\alph{my3})}{
\setlength{\parsep}{-1mm}
\setlength{\labelwidth}{8mm}
\setlength{\leftmargin}{14mm}
\usecounter{my3}}
}{\end{list}}

\title{\bf On the vertex cover number of $3-$uniform hypergraph}


 \author{Zhuo Diao ${}^{a}$\thanks{Corresponding author. E-mail: diaozhuo@amss.ac.cn} 
}
\date{
 ${}^a$ School of Statistics and Mathematics, Central University of Finance and Economics
Beijing 100081, China\\
} 



\maketitle

\begin{abstract}
Given a hypergraph $H(V,E)$, a set of vertices $S\subseteq V$ is a vertex cover if every edge has at least a vertex in $S$.
The vertex cover number is the minimum cardinality of a vertex cover, denoted by $\tau(H)$. In this paper,
we prove that for every $3-$uniform connected hypergraph $H(V,E)$, $\tau(H)\leq \frac{2m+1}{3}$ holds on where $m$ is the number of edges.
Furthermore, the equality holds on if and only if $H(V,E)$ is a hypertree with perfect matching.

\end{abstract}


\section{Introduction}

Given a hypergraph $H(V,E)$, a set of vertices $S\subseteq V$ is a vertex cover if every edge has at least a vertex in $S$.
The vertex cover number is the minimum cardinality of a vertex cover, denoted by $\tau(H)$. The vertex covering number is a key parameter in hypergraph theory. There are numerous literatures about this problem~\cite{berge1989}\cite{DSW1994}\cite{DK2016}\cite{OM2005}. A set of edges $A\subseteq E$ which are pairwise disjoint is a matching and the matching number is the maximum cardinality of a matching, denoted by $\nu(H)$. There are also numerous literatures about this problem~\cite{AKS1997}\cite{AK1999}. It is clear that $\nu(H)\leq \tau(H)$ always holds. If a hypergraph $H$ satisfies $\nu(H)= \tau(H)$, then we say Konig Property holds in $H$. In chapter $5$ of Hypergraphs~\cite{berge1989}, berge gives many conditions about Konig Property. He has proven that if a hypergraph $H$ has no odd cycle, then Konig Property holds in $H$, which says $\nu(H)= \tau(H)$ holds.


Packing and covering are so important because they are prime-dual parameters. Specially, an important category of packing and covering is the vertex cover and matching. In fact, for every hypergraph $H(V,E)$, it is easy to construct two integral programmes whose optimal values are $\tau(H)$ and $\nu(H)$. Furthermore, by relaxing integral constraints to linear constraints, there are two prime-dual programmes whose optimal values are $\tau^{\ast}(H)$ and $\nu^{\ast}(H)$. So $\tau(H) \ge \tau^{\ast}(H)  = \nu^{\ast}(H)  \ge \nu(H)$ holds on. $\tau(H)$ and $\nu(H)$ are prime-dual parameters.

Our contribution:

In this paper, we prove that for every $3-$uniform connected hypergraph $H(V,E)$, $\tau(H)\leq \frac{2m+1}{3}$ holds on where $m$ is the number of edges.
Furthermore, the equality holds on if and only if $H(V,E)$ is a hypertree with perfect matching.



\section{Hypergraphs}\label{sec:hypergraph}

In this section, we introduce the basic conceptions in hypergraph theory.\\

Let $H=(V,E)$ be a hypergraph with vertex set $V$ and edge set $E$. As in graph theory, we denote $n=|V|$ and $m=|E|$. For each $v\in V$, the {\em degree} $d(v)$ is the number of edges containing $v$. Let $k\in\mathbb Z_{>0}$ be a positive integer, hypergraph $H$ is called  {\em $k$-regular}  if $d(v)=k$ for each $v\in V$, and   {\em $k$-uniform} if $|e|=k$ for each $e\in E$. Hypergraph $H$ is   {\em linear} if $|e\cap f|\le1$ for any pair of distinct edges $e,f\in E$.\\

A vertex-edge alternating sequence $ v_{1}e_{1}v_{2}...v_{k}e_{k}v_{k+1}$ of $H$ is called a {\em path} (of {\em length} $k$) between $v_{1}$ and $v_{k+1}$ if $v_{1}, v_{2},..., v_{k+1}\in V$ are distinct, $ e_{1}, e_{2},..., e_{k}\in E$ are distinct, and  $\{v_{i},v_{i+1}\}\subseteq e_{i}$ for each $i\in [k]=\{1,\ldots,k\}$. Hypergraph $H$ is said to be {\em connected} if  there is a path between any pair of distinct vertices in $H$. A maximal connected subgraph of $H$  is called a {\em component} of $H$. \\

  A vertex-edge alternating sequence   $C= v_{1}e_{1}v_{2}e_{2}...v_{k}e_{k}v_{1}$, where $k\ge2$, is called a {\em cycle} (of length $k$) if $v_{1}, v_{2},..., v_{k}\in V$ are distinct, $ e_{1}, e_{2},..., e_{k}\in E$ are distinct, and  $\{v_{i},v_{i+1}\}\subseteq e_{i}$ for each $i\in [k]$, where $v_{k+1}=v_{1}$. {We call vertices $v_1,v_2,\ldots,v_k$  {\em join vertices} of $C$, and the other vertices  {\em non-join {vertices}} of $C$.}Hypergraph $H$ is said to be a {\em hyperforest} if $H$ is acyclic. Hypergraph $H$ is said to be a {\em hypertree} if $H$ is connected and acyclic.  \\

  For any $S\subseteq V$, we write $H\setminus S$  for the  subgraph of $H$ obtained from $H$ by deleting all vertices in $S$ and all edges  incident with some vertices in $S$.  For any $A\subseteq E$, we write $H\setminus A$  for the  subgraph of $H$ obtained from $H$ by deleting all edges in $A$ and keeping vertices.
  For any $S\subseteq V$, we write $H[S]$ for the subgraph of $H$ induced by the vertex set $S$.  For any $A\subseteq E$, we write $H[A]$ for the subgraph of $H$ induced by the edge set $A$.\\

Given a hypergraph $H(V,E)$, a set of vertices $S\subseteq V$ is a vertex cover if every edge has at least a vertex in $S$ which means $H\setminus S$ has no edges. The vertex cover number is the minimum cardinality of a vertex cover, denoted by $\tau(H)$. A set of edges $A\subseteq E$ is a matching if every two distinct edges have no common vertex. The matching number is the maximum cardinality of a matching, denoted by $\nu(H)$.

\section{The vertex cover number of $3-$uniform hypergraph}\label{sec:vcmaintheorem}

In this section, we will prove our main theorem as following:

\begin{theorem}
For every $3-$uniform connected hypergraph $H(V,E)$, $\tau(H)\leq \frac{2m+1}{3}$ holds on.
Furthermore, the equality holds on if and only if $H(V,E)$ is a hypertree with perfect matching.
\end{theorem}

In subsection $1$, we will prove for every $3-$uniform connected hypergraph $H(V,E)$, $\tau(H)\leq \frac{2m+1}{3}$ holds on. And in subsection $2$, we will prove the equality holds on if and only if $H(V,E)$ is a hypertree with perfect matching.

\subsection{General bounds}\label{sec:vcgeneralbounds}

In this subsection, we will prove the theorem as following:

\begin{theorem}
For every $3-$uniform connected hypergraph $H(V,E)$, $\tau(H)\leq \frac{2m+1}{3}$ holds on.
\end{theorem}

Before proving the theorem above, we will prove a series of lemmas which are very useful.

\begin{lemma}\label{connectn2m1}
For every $3-$uniform connected hypergraph $H(V,E)$, $n\leq 2m+1$ holds on.
\end{lemma}

\begin{proof}
We prove this lemma by induction on $m$. When $m=0$, $H(V,E)$ is an isolate vertex, $n\leq 2m+1$ holds on.
Assume this lemma holds on for $m\leq k$. When $m=k+1$, take arbitrarily one edge $e$ and consider the subgraph $H\setminus e$.
obviously, $H\setminus e$ has at most three components. Assume $H\setminus e$ has $p$ components $H_{i}(V_{i},E_{i})$ and $n_{i}=|V_{i}|, m_{i}=|E_{i}|$ for each $i\in\{1,...,p\}$. Then by our induction, $n_{i}\leq 2m_{i}+1$ holds on. So we have

\begin{equation}
n=n_{1}+...n_{p}\leq 2m_{1}+...2m_{p}+p=2(m-1)+p=2m+p-2\leq 2m+1
\end{equation}

By induction, we finish our proof.
\end{proof}

\begin{lemma}\label{connectntree2m1}
For every $3-$uniform connected hypergraph $H(V,E)$, $n=2m+1$ if and only if $H$ is a hypertree.
\end{lemma}

\begin{proof}
sufficiency: if $H$ is a hypertree, we prove $n=2m+1$ by induction on $m$. When $m=0$, $H(V,E)$ is an isolate vertex, $n=2m+1$ holds on.
Assume this lemma holds on for $m\leq k$. When $m=k+1$, take arbitrarily one edge $e$ and consider the subgraph $H\setminus e$.
Because $H$ is a hypertree, $H\setminus e$ has exactly three components, denoted by $H_{i}(V_{i},E_{i})$ and $n_{i}=|V_{i}|, m_{i}=|E_{i}|$ for each $i\in\{1,2,3\}$. Because every component is a hypertree, thus by our induction, $n_{i}= 2m_{i}+1$ holds on. So we have

\begin{equation}
n=n_{1}+n_{2}+n_{3}=2m_{1}+2m_{2}+2m_{3}+3=2(m-1)+3=2m+1
\end{equation}

By induction, we finish the sufficiency proof.\\

necessity: We prove by contradiction. If $H$ is not a hypertree, $H$ contain a cycle $C$. Take arbitrarily one edge $e$ in $C$ and consider the subgraph $H\setminus e$. obviously, $H\setminus e$ has at most two components. Assume $H\setminus e$ has $p$ components $H_{i}(V_{i},E_{i})$ and $n_{i}=|V_{i}|, m_{i}=|E_{i}|$ for each $i\in\{1,...,p\}$. Then by lemma~\ref{connectn2m1}, $n_{i}\leq 2m_{i}+1$ holds on. So we have

\begin{equation}
n=n_{1}+...n_{p}\leq 2m_{1}+...2m_{p}+p=2(m-1)+p=2m+p-2\leq 2m< 2m+1
\end{equation}

which is a contradiction with $n=2m+1$. Thus $H$ is a hypertree and we finish our necessity proof.

\end{proof}

In chapter $5$ of Hypergraphs~\cite{berge1989}, berge has proven that if a hypergraph $H$ has no odd cycle, then Konig Property holds in $H$, which says $\nu(H)= \tau(H)$ holds. Thus next theorem is obvious.

\begin{theorem}\label{the:Konig}
$H(V,E)$ is a hypertree, then $\tau(H)=\nu(H)$
\end{theorem}



\begin{definition}\label{def:minimalcycle}
A cycle $C= v_{1}e_{1}v_{2}e_{2}...v_{k}e_{k}v_{1}$, where $k\ge2$, is a minimal cycle if any two non-adjacent edges have no common vertex, that is
for each $i,j\in\{1,...,k\}, 1<|j-i|<k-1$, we have $e_{i}\cap e_{j}=\phi$.
\end{definition}

\begin{lemma}\label{lem:commonvertex}
A cycle $C$ is not minimal, then there exists two distinct cycles $C_{1}$ and $C_{2}$ with $C_{1}\subseteq C, C_{2}\subseteq C, |C_{1}|<|C|,|C_{2}|<|C|,|C_{1}|+|C_{2}|\leq |C|+2$. Furthermore, $C_{1}$ and $C_{2}$ have common vertices(edges).
\end{lemma}

\begin{proof}
A cycle $C= v_{1}e_{1}v_{2}e_{2}...v_{k}e_{k}v_{1}$ is not minimal, then according to definition~\ref{def:minimalcycle}, there exists $i,j\in\{1,...,k\}, 1<|j-i|<k-1$, we have $e_{i}\cap e_{j}\neq \phi$. Thus we know $k\geq 4$. Assume $i<j$ without generality and take arbitrarily $v\in e_{i}\cap e_{j}$. \\

Case $1$: $v\in \{v_{i},v_{i+1},v_{j},v_{j+1}\}$. Assume $v=v_{i}$ without generality, let us take $C_{1}=ve_{i}v_{i+1}e_{i+1}...v_{j}e_{j}v$ and $C_{2}=ve_{j}v_{j+1}...v_{i-1}e_{i-1}v$, then $C_{1}\subseteq C, C_{2}\subseteq C, |C_{1}|<|C|,|C_{2}|<|C|,|C_{1}|+|C_{2}|= |C|+1$. $C_{1}$ and $C_{2}$ have common vertices(edges).\\

Case $2$: $v\in \{v_{1},...,v_{k}\}-\{v_{i},v_{i+1},v_{j},v_{j+1}\}$. Assume $v=v_{t}, i+1<t<j$ without generality, let us take $C_{1}=ve_{t}v_{t+1}...v_{j}e_{j}v$ and $C_{2}=ve_{j}v_{j+1}...v_{t-1}e_{t-1}v$, then $C_{1}\subseteq C, C_{2}\subseteq C, |C_{1}|<|C|,|C_{2}|<|C|,|C_{1}|+|C_{2}|= |C|+1$. $C_{1}$ and $C_{2}$ have common vertices(edges).\\

Case $3$: $v\not\in \{v_{1},...,v_{k}\}$. Let us take $C_{1}=ve_{i}v_{i+1}e_{i+1}...v_{j}e_{j}v$ and $C_{2}=ve_{j}v_{j+1}...v_{i}e_{i}v$, then $C_{1}\subseteq C, C_{2}\subseteq C, |C_{1}|<|C|,|C_{2}|<|C|,|C_{1}|+|C_{2}|= |C|+2$. $C_{1}$ and $C_{2}$ have common vertices(edges).\\

\end{proof}

\begin{lemma}\label{lem:minimalcycle}
$H(V,E)$ have two distinct cycles $C_{1}$ and $C_{2}$ with common vertices(edges), then $H(V,E)$ have two distinct minimal cycles $C'_{1}$ and $C'_{2}$ with common vertices(edges).
\end{lemma}

\begin{proof}
If $C_{1}$ and $C_{2}$ are both minimal, then we take $C'_{1}=C_{1}$ and $C'_{2}=C_{2}$. If $C_{1}$ or $C_{2}$ is not minimal, assume $C_{1}$ is not minimal
without generality, according to lemma~\ref{lem:commonvertex}, there exists two distinct cycles $C'_{1}$ and $C'_{2}$ with $C'_{1}\subseteq C_{1}, C'_{2}\subseteq C_{1}, |C'_{1}|<|C_{1}|,|C'_{2}|<|C_{1}|,|C'_{1}|+|C'_{2}|\leq |C_{1}|+2$. Furthermore, $C'_{1}$ and $C'_{2}$ have common vertices(edges). If $C'_{1}$ or $C'_{2}$ is not minimal, we can repeat this process. Because every time $|C'_{1}|<|C_{1}|$, after finite steps, this process is terminated and we have two distinct minimal cycles $C'_{1}$ and $C'_{2}$ with common vertices(edges).
\end{proof}

\begin{lemma}\label{lem:2d2}
$H(V,E)$ have two distinct cycles $C_{1}$ and $C_{2}$ with common vertices, then there exists $v\in H$ such that $H\setminus v$ has at most $2d(v)-2$ components.
\end{lemma}

\begin{proof}
According to lemma~\ref{lem:minimalcycle}, $H(V,E)$ have two distinct minimal cycles $C_{1}$ and $C_{2}$ with common vertices. \\

Case $1$: $C_{1}$ is not linear and $|C_{1}|\geq 4$. Because $C_{1}$ is minimal, $C_{1}$ must be shown in Figure~\ref{graphs-1}. So $H\setminus v$ has at most  $2+2[d(v)-2]=2d(v)-2$ components.\\

\begin{figure}[h]
\begin{minipage}[h]{0.5\linewidth}
\centering
\includegraphics[width=2in]{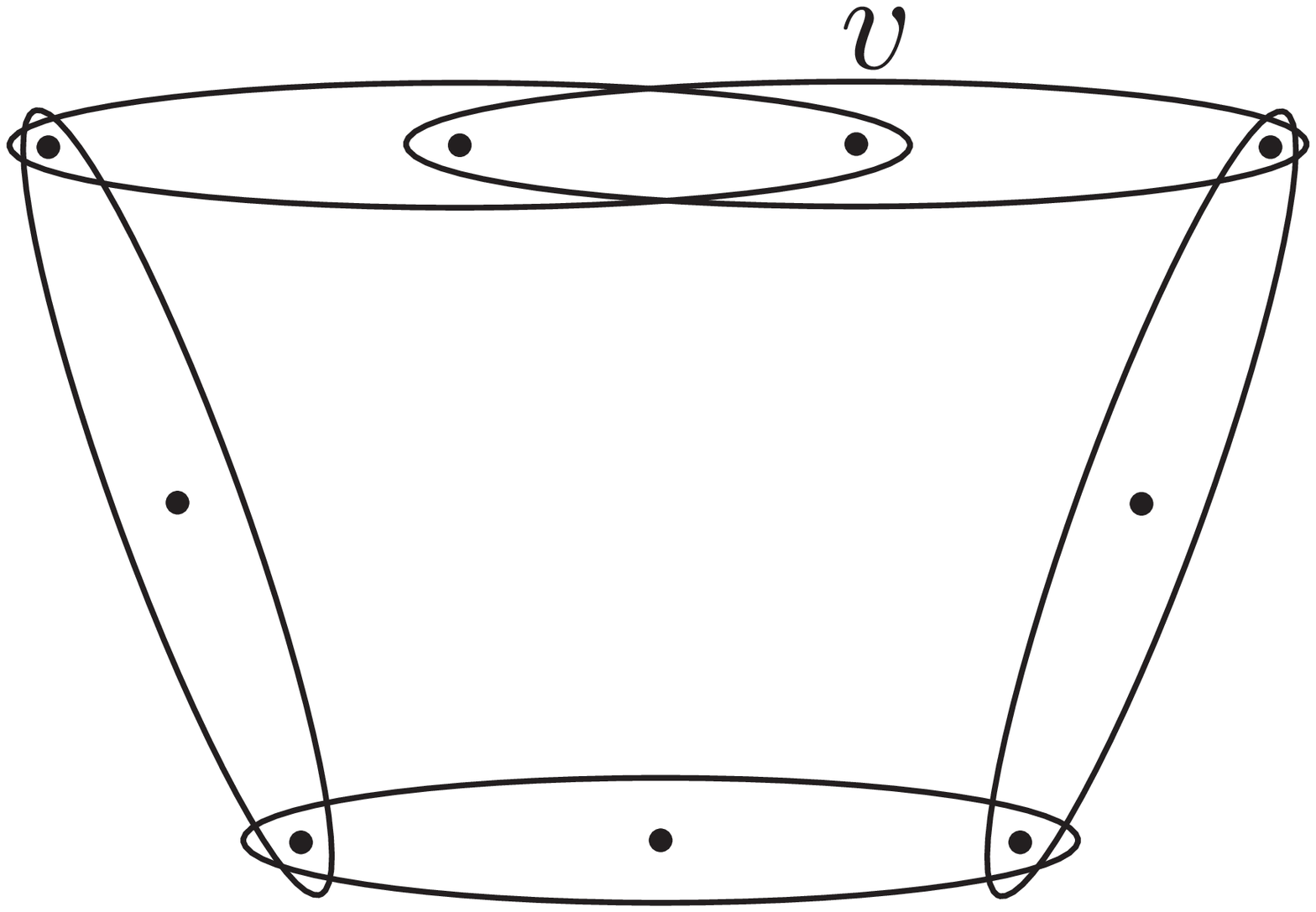}
\caption{$C_{1}$}
\label{graphs-1}
\end{minipage}%
\begin{minipage}[h]{0.5\linewidth}
\centering
\includegraphics[width=2in]{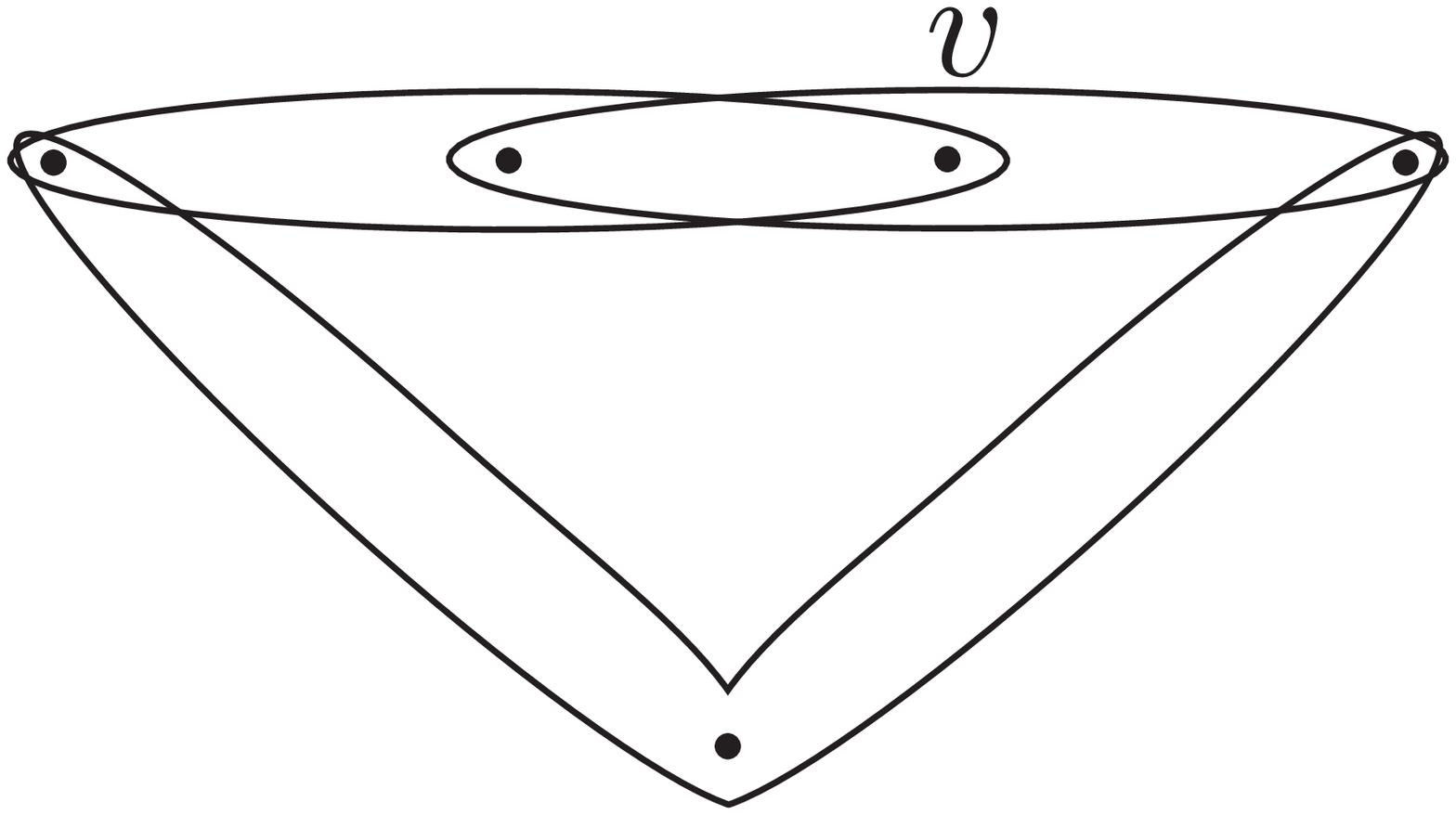}
\caption{$C_{1}$}
\label{graphs-2}
\end{minipage}
\end{figure}

Case $2$: $C_{1}$ is not linear and $|C_{1}|=3$. Assume $C_{1}=v_{1}e_{1}v_{2}e_{2}v_{3}e_{3}v_{1}$, here $e_{1}=\{v_{1},u_{1},v_{2}\},e_{2}=\{v_{2},u_{2},v_{3}\},e_{3}=\{v_{3},u_{3},v_{1}\}$. \\

If $e_{1}\cap e_{2}\cap e_{3}\neq \emptyset$, we know $C_{1}$ has at most $5$ vertices.(If $v_{1},v_{2},v_{3},u_{1},u_{2},u_{3}$ are all distinct, there must be $e_{1}\cap e_{2}\cap e_{3}= \emptyset$). So let us pick arbitrarily $u\in e_{1}\cap e_{2}\cap e_{3}$ and $H\setminus v$ has at most  $4+2[d(v)-3]=2d(v)-2$ components.\\

If $e_{1}\cap e_{2}\cap e_{3}\neq \emptyset$, $C_{1}$ must be shown in Figure~~\ref{graphs-2}. So $H\setminus v$ has at most  $2+2[d(v)-2]=2d(v)-2$ components.\\

Case $3$: $C_{1}$ and $C_{2}$ are both linear or $2-$ cycles. Assume $C_{1}=v_{1}e_{1}v_{2}e_{2}...e_{k}v_{1}$ and $C_{2}=\tilde{v_{1}}\tilde{e_{1}}\tilde{v_{2}}\tilde{e_{2}}...\tilde{e_{t}}\tilde{v_{1}}$, here $k\geq 2, t\geq 2$. \\

$1.$ $C_{1}$ and $C_{2}$ have no common edges. Now because $C_{1}$ and $C_{2}$ have common vertices, let us pick arbitrarily a common vertex $v\in C_{1}\cap C_{2}$. \\

$(a)$. $v$ is a join-vertex of both $C_{1}$ and $C_{2}$, assume $v=v_{2}=\tilde{v_{2}}$. Now we have $e_{1}=\{v_{1},u_{1},v_{2}\},e_{2}=\{v_{2},u_{2},v_{3}\},\tilde{e_{1}}=\{\tilde{v_{1}},\tilde{u_{1}},\tilde{v_{2}}\},
\tilde{e_{2}}=\{\tilde{v_{2}},\tilde{u_{2}},\tilde{v_{3}}\}$.\\

It is easy to see in $H\setminus v$, there exists $v_{1}-v_{3}$ path and $\tilde{v_{1}}-\tilde{v_{3}}$ path. So $H\setminus v$ has at most  $6+2[d(v)-4]=2d(v)-2$ components.\\

$(b)$. $v$ is a join-vertex of $C_{1}$ and a non-join vertex of $C_{2}$, assume $v=v_{2}=\tilde{u_{1}}$. Now we have $e_{1}=\{v_{1},u_{1},v_{2}\},e_{2}=\{v_{2},u_{2},v_{3}\},\tilde{e_{1}}=\{\tilde{v_{1}},\tilde{u_{1}},\tilde{v_{2}}\}$.\\

It is easy to see in $H\setminus v$, there exists $v_{1}-v_{3}$ path and $\tilde{v_{1}}-\tilde{v_{2}}$ path. So $H\setminus v$ has at most  $4+2[d(v)-3]=2d(v)-2$ components.\\

$(c)$. $v$ is a non-join vertex of both $C_{1}$ and $C_{2}$, assume $v=u_{1}=\tilde{u_{1}}$. Now we have $e_{1}=\{v_{1},u_{1},v_{2}\},\tilde{e_{1}}=\{\tilde{v_{1}},\tilde{u_{1}},\tilde{v_{2}}\}$.\\

It is easy to see in $H\setminus v$, there exists $v_{1}-v_{2}$ path and $\tilde{v_{1}}-\tilde{v_{2}}$ path. So $H\setminus v$ has at most  $2+2[d(v)-2]=2d(v)-2$ components.\\

$2.$ $C_{1}$ and $C_{2}$ have common edges. Obviously there exists two adjacent edges in $C_{1}$ such that one edge belongs to $C_{2}$ and the other edge
does not belong to $C_{2}$. We can assume $e_{1}\in C_{2}$ and $e_{2}\not\in C_{2}$, furthermore, $e_{1}=\tilde{e_{1}}$. \\

$(a)$ $v_{2}=\tilde{v_{2}}$ Now we have $e_{1}=\{v_{1},u_{1},v_{2}\},e_{2}=\{v_{2},u_{2},v_{3}\},\tilde{e_{1}}=\{\tilde{v_{1}},\tilde{u_{1}},\tilde{v_{2}}\},
\tilde{e_{2}}=\{\tilde{v_{2}},\tilde{u_{2}},\tilde{v_{3}}\}$.\\

It is easy to see in $H\setminus v_{2}(\tilde{v_{2}})$, there exists $v_{1}-v_{3}$ path and $\tilde{v_{1}}-\tilde{v_{3}}$ path. Combined with $e_{1}=\tilde{e_{1}}$, so $H\setminus v$ has at most  $4+2[d(v)-3]=2d(v)-2$ components.\\

$(b)$ $v_{2}=\tilde{u_{1}}$ Now we have $e_{1}=\{v_{1},u_{1},v_{2}\},e_{2}=\{v_{2},u_{2},v_{3}\},\tilde{e_{1}}=\{\tilde{v_{1}},\tilde{u_{1}},\tilde{v_{2}}\}$.\\

It is easy to see in $H\setminus v_{2}(\tilde{u_{1}})$, there exists $v_{1}-v_{3}$ path and $\tilde{v_{1}}-\tilde{v_{2}}$ path. Combined with $e_{1}=\tilde{e_{1}}$, so $H\setminus v$ has at most  $2+2[d(v)-2]=2d(v)-2$ components.\\

$(c)$ $v_{2}=\tilde{v_{1}}$ Now we have $e_{1}=\{v_{1},u_{1},v_{2}\},e_{2}=\{v_{2},u_{2},v_{3}\},\tilde{e_{1}}=\{\tilde{v_{1}},\tilde{u_{1}},\tilde{v_{2}}\},
\tilde{e_{t}}=\{\tilde{v_{t}},\tilde{u_{t}},\tilde{v_{1}}\}$.\\

It is easy to see in $H\setminus v_{2}(\tilde{v_{1}})$, there exists $v_{1}-v_{3}$ path and $\tilde{v_{t}}-\tilde{v_{2}}$ path. Combined with $e_{1}=\tilde{e_{1}}$, so $H\setminus v$ has at most  $4+2[d(v)-3]=2d(v)-2$ components.\\

Above all, in whatever case, there always exists $v\in V$ such that $H\setminus v$ has at most $2d(v)-2$ components.

\end{proof}

Now we will prove our main theorem as following:

\begin{theorem}\label{the:mainhyper}
Let $H(V,E)$ be a connected $3$-uniform  hypergraph. Then $\tau(H)\leq \frac{2m+1}{3}$ holds on where $m$ is the number of edges.
\end{theorem}

\begin{proof}

We prove this theorem by contradiction. Let us take out the counterexample $H=(V,E)$ with minimum edges, thus $\tau(H)> \frac{2m+1}{3}$. We have a series of claims:

\begin{claim}\label{cla:cycle}
$H=(V,E)$ has some cycles
\end{claim}

If $H=(V,E)$ is acyclic, according to Lemma~\ref{connectn2m1} and Theorem~\ref{the:Konig}, we have next inequalities, a contradiction.

\[\frac{2m+1}{3}< \tau(H)=\nu(H)\leq \frac{n}{3}=\frac{2m+1}{3}\]

\begin{claim}\label{cla:componentcondition}
For each vertex $v\in \mathcal V$, $\mathcal H\setminus v$ has at least $2d(v)-1$ components.
\end{claim}

If there exist a vertex $v\in V$, $\mathcal H\setminus v$ has $k$ components with $k\leq 2d(v)-2$. Assume $H_{i}, i\in [k]$ are $k$ components of $H\setminus v$ and $m_{i}$ is the number of edges in $H_{i}, i\in [k]$. We have next inequalities:

\[\tau(H)\leq \tau(H\setminus v)+1= \sum_{i\in [k]}\tau(H_{i})+1\leq \sum_{i\in [k]}\frac{2m_{i}+1}{3}+1=\frac{2\sum_{i\in [k]}m_{i}+k}{3}+1\]
\[=\frac{2[m-d(v)]+k}{3}+1\leq \frac{2[m-d(v)]+2d(v)-2}{3}+1=\frac{2m+1}{3}\]

where the second inequality holds on because $H$ is the counterexample with minimum edges. This is a contradiction with $\tau(H)> \frac{2m+1}{3}$.Combined with lemma~\ref{lem:2d2}, we have next claim instantly.

\begin{claim}\label{cla:cyclevertexdisjoint}
Every two distinct cycles in $H$ are vertex-disjoint.
\end{claim}

According to claim, it is easy to know every cycle of $H=(V,E)$ is $2-$ cycle or linear minimal and every two distinct cycles are joined together through a hypertree. Next we will construct a tree $T(V_{1}\cup V_{2},E_{T})$ by $H=(V,E)$. $V_{1}$ denotes the set of cycles in $H=(V,E)$, $V_{2}$ denotes the set of hypertrees in $H=(V,E)$ and for each $v_{1}\in V_{1}, v_{2}\in V_{2}, e(v_{1},v_{2})\in E_{T}$ if and only if the cycle and the hypertree are connected. See Figure~\ref{tree1} for an illustration.

\begin{figure}[h]
\begin{center}
\includegraphics[scale=0.48]{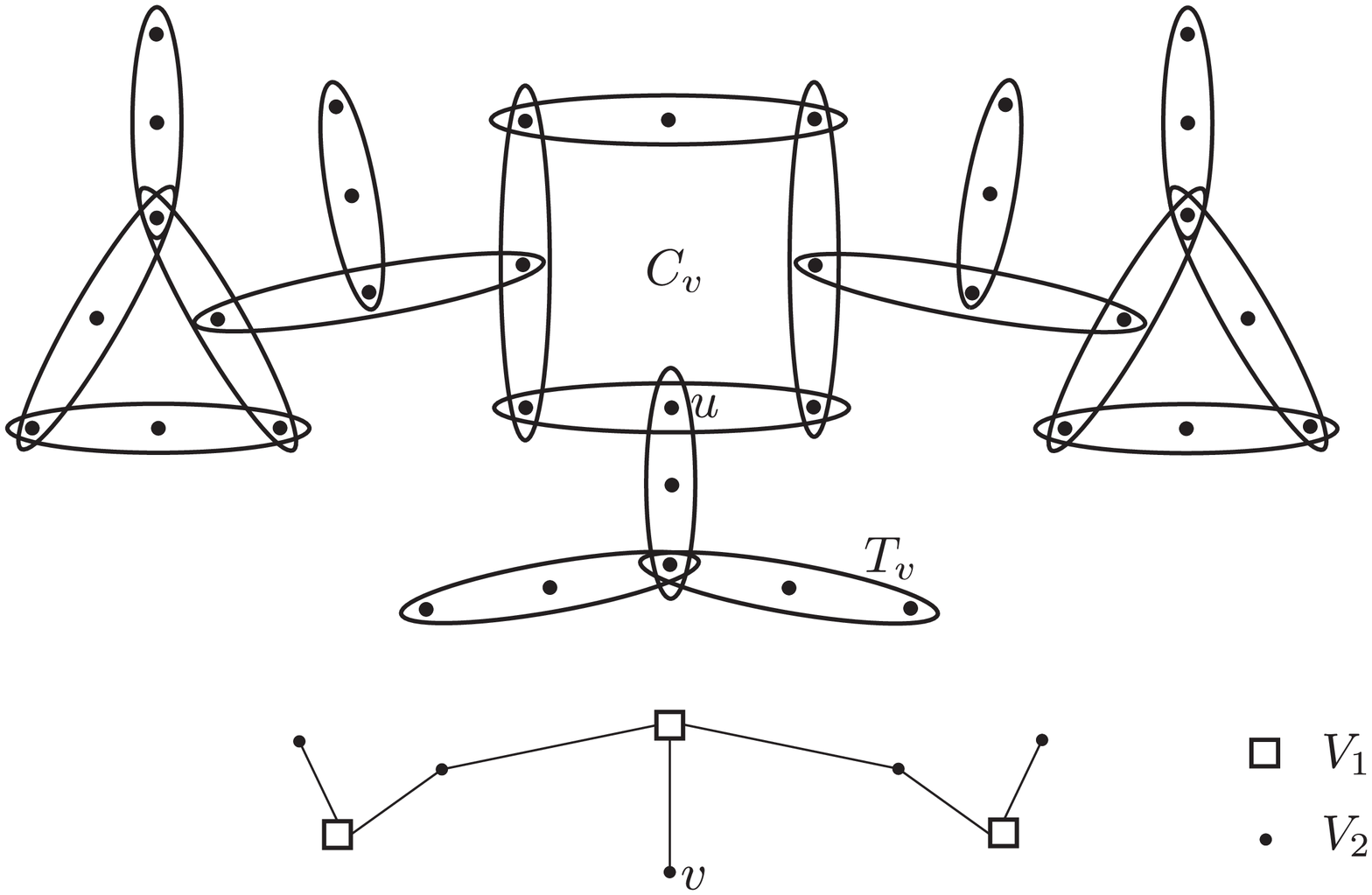}
\caption{\label{tree1}}
\end{center}
\end{figure}

$(E_{1},E_{2})$ is an nonempty partition of $E$, $H_{1}$ is edge-induced subhypergraph of $H$ by $E_{1}$, $H_{2}$ is edge-induced subhypergraph of $H$ by $E_{2}$. $m_{1}$ is the number of edges in $H_{1}$ and $m_{2}$ is the number of edges in $H_{2}$. We have next important claim:

\begin{claim}\label{cla:subhypergraph}
If $H_{1}$ and $H_{2}$ are both connected, then $\tau(H_{1})= \frac{2m_{1}+1}{3}, \tau(\mathcal H_{2})= \frac{2m_{2}+1}{3}$
\end{claim}

This claim is instant by the inequalities below:

\[\frac{2m+1}{3}< \tau(H)\leq \tau(H_{1})+\tau(H_{2})\leq \frac{2m_{1}+1}{3}+\frac{2m_{2}+1}{3}=\frac{2m+2}{3}\]

Next let us consider a leaf $v$ of $T(V_{1}\cup V_{2},E_{T})$ and the leaf corresponds to a subhypergraph, which is a minimal cycle or a hypertree of $H$. We can take the edges in this subhypergraph as $E_{1}$ and other edges as $E_{2}$. Because $v$ is a leaf of $T(V_{1}\cup V_{2},E)$, we know $H_{1}$ and $H_{2}$ are both connected. According to claim~\ref{cla:subhypergraph}, we have $\tau(H_{1})= \frac{2m_{1}+1}{3}, \tau(H_{2})= \frac{2m_{2}+1}{3}$. For any cycle $C$, $\tau(C)\leq \frac{m_{c}+1}{2}<\frac{2m_{c}+1}{3}$. Thus the leaf must correspond to a hypertree. Let us denote the hypertree as $T_{v}$. Because $v$ is a leaf of $T(V_{1}\cup V_{2},E)$, $T_{v}$ is connected with an unique cycle $C_{v}$. We assume $T_{v}$ and $C_{v}$ are connected together through the vertex $u$. See Figure~\ref{tree1} for an illustration.\\

In $T_{v}$, among these vertices with largest distance from $u$, we pick arbitrarily one, denoted as $w$. Next we will prove the distance $d(u,w)=1$.

\begin{claim}\label{cla:distance1}
 $d(u,w)=1$.
\end{claim}

If the distance $d(u,w)\geq 2$, we have a subhypergraph as shown in Figure~\ref{tree2}. Because $w$ is the farthest vertex from $u$ in $T_{v}$, there must be $d(w)=d(w_{3})=1$. We can take the edges incident with $w_{1}$ or $w_{2}$ as $E_{1}$ and other edges as $E_{2}$. It is easy to know $H_{1}$ and $H_{2}$ are both connected. According to claim~\ref{cla:subhypergraph}, we have $\tau(H_{1})= \frac{2m_{1}+1}{3}, \tau(H_{2})= \frac{2m_{2}+1}{3}$. When $|E_{1}|=2$, we have $d(w_{1})=1, d(w_{2})=2$ and $w_{2}$ is a vertex cover for $E_{1}$, thus $\tau(H_{1})=1< \frac{2\times2+1}{3}=\frac{2m_{1}+1}{3}$, a contradiction. When $|E_{1}|\geq 3$, we have $d(w_{1})>1$ or $d(w_{2})>2$ and $w_{1}, w_{2}$ are a vertex cover for $\mathcal E_{1}$, thus $\tau(\mathcal H_{1})\leq 2< \frac{2\times3+1}{3}\leq \frac{2m_{1}+1}{3}$, also a contradiction. See Figure~\ref{tree2} for an illustration. Above all, in whatever case, there is always a contradiction, thus our assumption doesn't hold on and $d(u,w)=1$.

\begin{claim}\label{cla:oneedge}
$T_{v}$ is one edge
\end{claim}

According to claim~\ref{cla:distance1}, we know every edge in $T_{v}$ is incident with the vertex $u$. We can take the edges in $T_{v}$ as $E_{1}$ and other edges as $E_{2}$. It is easy to know $H_{1}$ and $H_{2}$ are both connected. According to claim~\ref{cla:subhypergraph}, we have $\tau(H_{1})= \frac{2m_{1}+1}{3}, \tau(H_{2})= \frac{2m_{2}+1}{3}$. When $|E_{1}|\geq 2$, we have $u$ is a vertex cover for $E_{1}$, thus $\tau(H_{1})=1< \frac{2\times2+1}{3}\leq \frac{2m_{1}+1}{3}$, a contradiction. Thus $|E_{1}|=1$ and $T_{v}$ is one edge. See Figure~\ref{tree3}$(a)$ for an illustration. \\

Now we know $T_{v}$ is one edge incident with $u$. We will finish our proof by the following two cases:\\

Case $1$: $u$ is a nonjoin-vertex of $C_{v}$, as shown in Figure~\ref{tree3}$(b)$. We can take the edges incident with $u$ as $E_{1}$ and other edges as $E_{2}$. It is easy to know $H_{1}$ and $H_{2}$ are both connected. According to claim~\ref{cla:subhypergraph}, we have $\tau(H_{1})= \frac{2m_{1}+1}{3}, \tau(H_{2})= \frac{2m_{2}+1}{3}$. But $|E_{1}|= 2$, we have $u$ is a vertex cover for $E_{1}$, thus $\tau(H_{1})=1< \frac{2\times2+1}{3}= \frac{2m_{1}+1}{3}$, a contradiction. See Figure~\ref{tree3}$(b)$ for an illustration.\\

Case $2$: $u$ is a join-vertex of $C_{v}$, as shown in Figure~\ref{tree3}$(c)$. We can take the edges $\{e,e'\}$ incident with $u$ as $E_{1}$ and other edges as $E_{2}$. It is easy to know $H_{1}$ is connected and $H_{2}$ has at most two components. If $H_{2}$ is connected, according to claim~\ref{cla:subhypergraph}, we have $\tau(H_{1})= \frac{2m_{1}+1}{3}, \tau(H_{2})= \frac{2m_{2}+1}{3}$. But $|E_{1}|= 2$, we have $u$ is a vertex cover for $E_{1}$, thus $\tau(H_{1})=1< \frac{2\times2+1}{3}= \frac{2m_{1}+1}{3}$, a contradiction. If $H_{2}$ has two components, denoted as $H_{3}$ and $H_{4}$, we have next inequalities, also a contradiction. See Figure~\ref{tree3}$(c)$ for an illustration.

\[\frac{2m+1}{3}< \tau(H)\leq \tau(H_{1})+\tau(H_{2})=\tau(H_{1})+\tau(H_{3})+\tau(H_{4})\]
\[\leq 1+\frac{2m_{3}+1}{3}+\frac{2m_{4}+1}{3}=1+\frac{2m_{2}+2}{3}=1+\frac{2(m-2)+2}{3} =\frac{2m+1}{3}\]

Above all, in whatever case, there is always a contradiction, thus our assumption doesn't hold on and the theorem is proven.

\end{proof}

\begin{figure}[h]
\begin{center}
\includegraphics[scale=0.48]{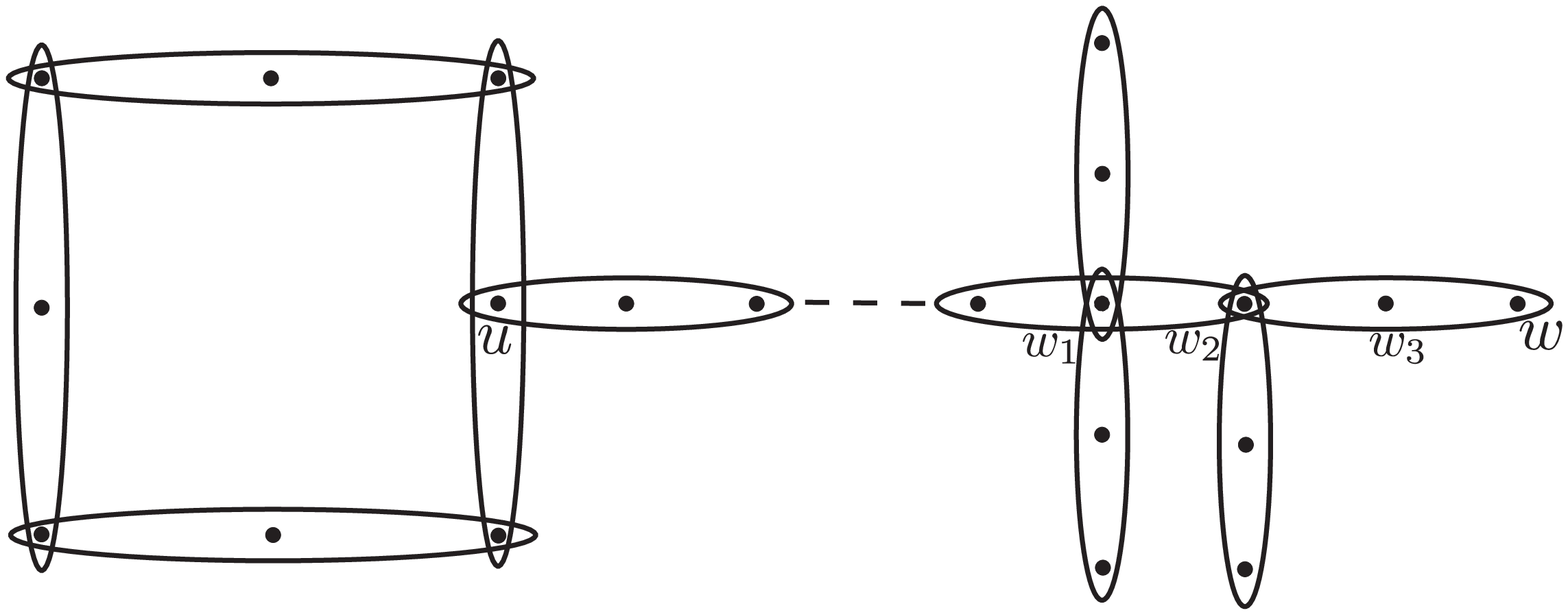}
\caption{\label{tree2}}
\end{center}
\end{figure}

\begin{figure}[h]
\begin{center}
\includegraphics[scale=0.48]{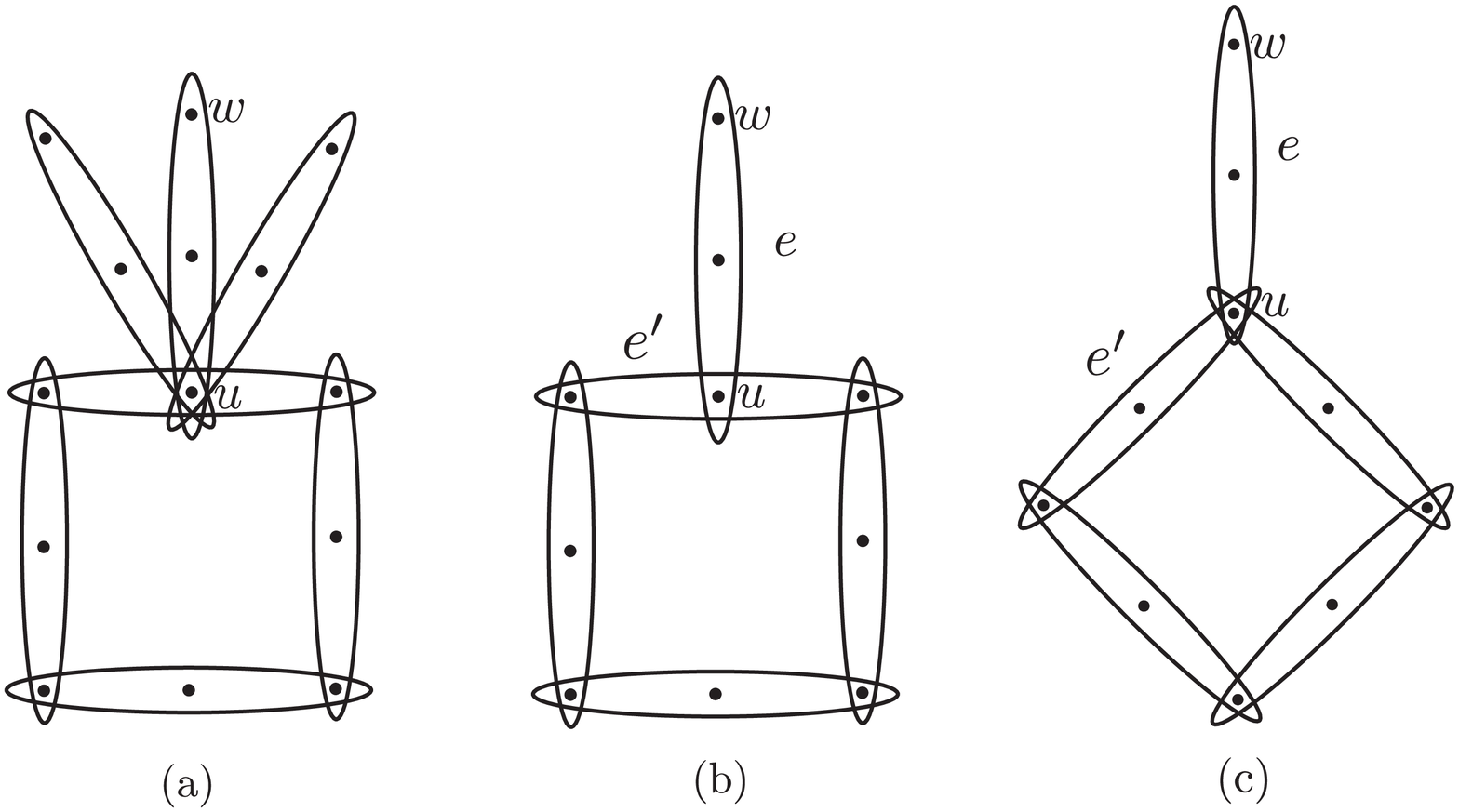}
\caption{\label{tree3}}
\end{center}
\end{figure}

\subsection{Extremal hypergraphs}\label{sec:extremalhyper}

In this subsection, we will prove the theorem as following:

\begin{theorem}\label{the:mainhyper}
Let $H(V,E)$ be a connected $3$-uniform  hypergraph. Then $\tau(H)\leq \frac{2m+1}{3}$ holds on where $m$ is the number of edges.
Furthermore, $\tau(H)= \frac{2m+1}{3}$ if and only if $H(V,E)$ is a hypertree with perfect matching.
\end{theorem}

\begin{proof}

Sufficiency: If $H(V,E)$ is a hypertree with perfect matching, then according to lemma~\ref{connectntree2m1} and theorem~\ref{the:Konig}, we have next equalities:

\[\tau(H)=\nu(H)=\frac{n}{3}=\frac{2m+1}{3}\]

Necessity: When $\tau(H)= \frac{2m+1}{3}$, we need to prove $H(V,E)$ is a hypertree with perfect matching. It is enough to prove $H(V,E)$ is acyclic.
Actually, if $H(V,E)$ is acyclic, according to lemma~\ref{connectn2m1} and theorem~\ref{the:Konig}, we have next inequalities:

\[\tau(H)=\nu(H)\leq \frac{n}{3}=\frac{2m+1}{3}\]

Combined with $\tau(H)= \frac{2m+1}{3}$, we have next equalities, which says $H(V,E)$ is a hypertree with perfect matching.

\[\tau(H)=\nu(H)=\frac{n}{3}=\frac{2m+1}{3}\]

By contradiction, let us take out a counterexample $H(V,E)$ with minimum edges. Then $\tau(H)= \frac{2m+1}{3}$ and $H(V,E)$ contains cycles. We have a series of claims:

\begin{claim}\label{cla:commonedge}
Every two distinct cycles in $H$ share common edges.
\end{claim}

Actually, for every two distinct cycles $C_{1}$ and $C_{2}$, if $E(C_{1})\cap E(C_{2})=\emptyset$,
then we can partition the set of edges $E(H)$ into two parts $E(H_{1})$ and $E(H_{2})$ such that $E(C_{1})\subseteq E(H_{1}),E(C_{2})\subseteq E(H_{2})$
and the edge-induced subhypergraphs $H_{1}$ and $H_{2}$ are both connected. Because $H(V,E)$ is a counterexample with minimum edges, we have next inequalities,
a contradiction with the assumption $\tau(H)= \frac{2m+1}{3}$.

\[\tau(H_{1})\leq\frac{2m_{1}}{3},\tau(H_{2})\leq\frac{2m_{2}}{3}\Rightarrow \tau(H)\leq \tau(H_{1})+\tau(H_{2})\leq\frac{2m_{1}}{3}+\frac{2m_{2}}{3}=\frac{2m}{3}\]

Let us take out a shortest cycle $C$. Because $\tau(C)\leq \frac{m_{c}+1}{2}<\frac{2m_{c}+1}{3}$, we know $E(H)\setminus E(C)\neq \emptyset$.
Furthermore,according to claim, we know $E(H)\setminus E(C)$ induces some hypertrees. The next claim is essential.

\begin{claim}\label{cla:singleedge}
Every hypertree induced by $E(H)\setminus E(C)$ must be an edge.
\end{claim}

We assume there exists a hypertree $T$ with $|E(T)|\geq 2$. Then let us take arbitrarily a vertex $v\in T\cap C$ and denote the farthest vertex from $v$ in $T$ as $v'$. We have next two cases.\\

Case $1$: distance $(v,v')=1$ in $T$, we have a partial structure in Figure~\ref{graphs-3}. Now we can take $\{e_{1},e_{2}\}$ as $E_{1}$ and other edges as $E_{2}$. It is easy to know the edge-induced subhypergraphs $H_{1}$ and $H_{2}$ are both connected. Thus we have next inequalities,which is contradiction with $\tau(H)= \frac{2m+1}{3}$.

\[\tau(H)\leq \tau(H_{1})+\tau(H_{2})\leq 1+\frac{2(m-2)+1}{3}=\frac{2m}{3}<\frac{2m+1}{3}\]

\begin{figure}[h]
\begin{minipage}[h]{0.5\linewidth}
\centering
\includegraphics[width=2in]{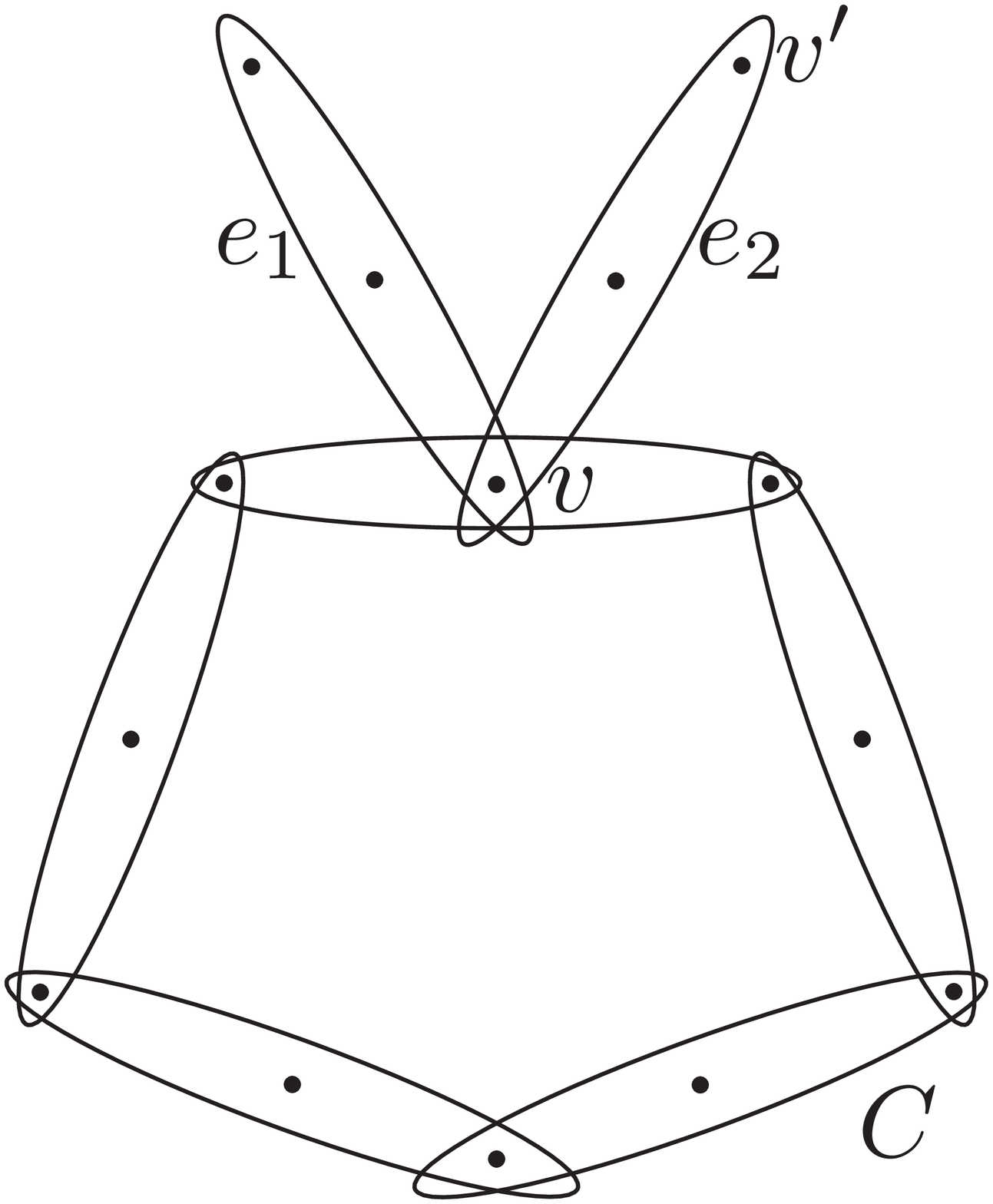}
\caption{}
\label{graphs-3}
\end{minipage}%
\begin{minipage}[h]{0.5\linewidth}
\centering
\includegraphics[width=2in]{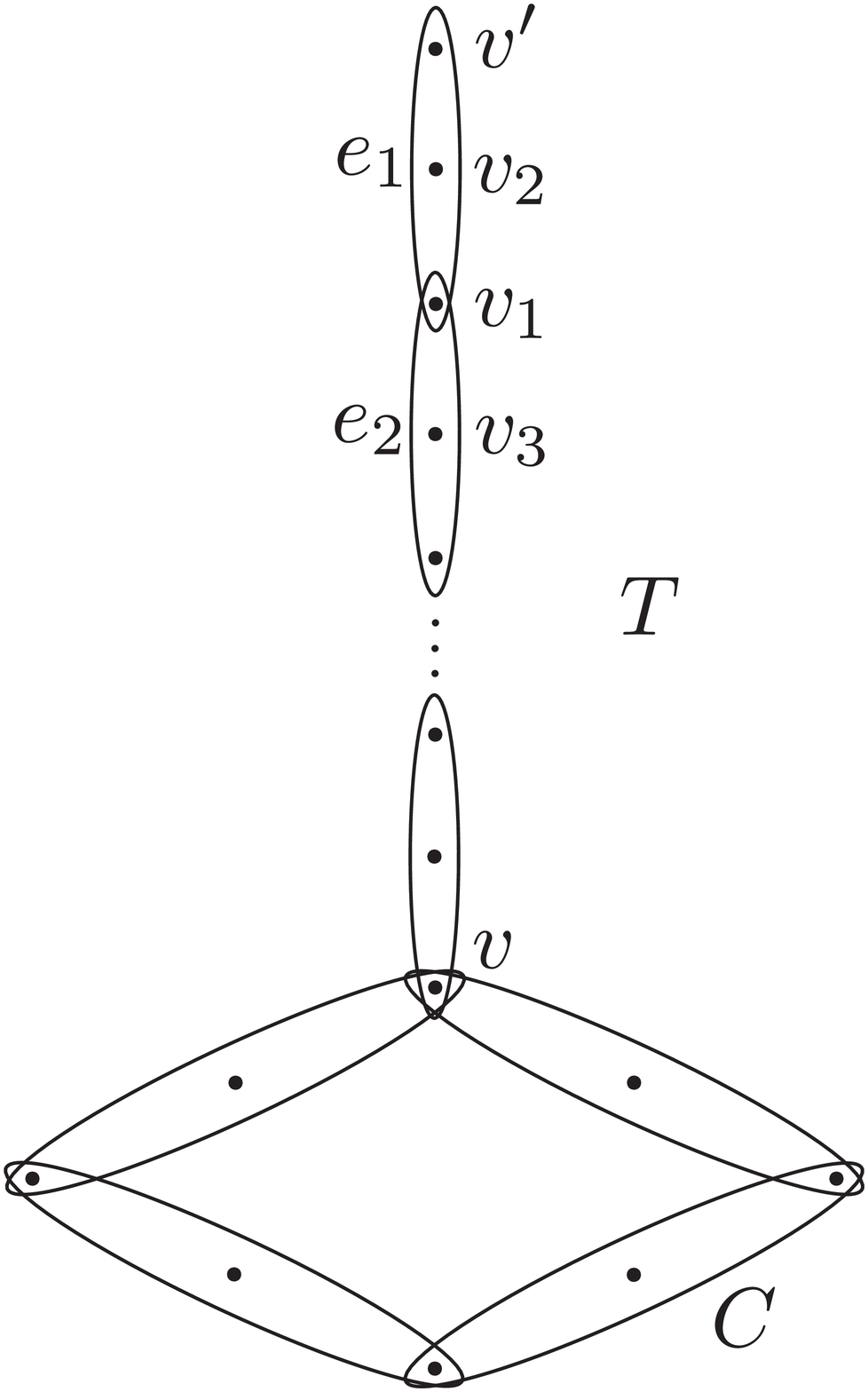}
\caption{}
\label{graphs-4}
\end{minipage}
\end{figure}

Case $2$: distance $(v,v')\geq 2$ in $T$, we have a partial structure in Figure~\ref{graphs-4}. Because $v'$ is the farthest vertex from $v$ in $T$, there must be $d(v')=d(v_{2})=1$ in $T$, which says $e_{1}$ is the unique edge containing $v'$ or $v_{2}$ in $T$.  We can take the edges incident with $v_{1}$ in $T$ as $E_{1}$ and other edges as $E_{2}$. It is easy to know $H_{1}$ is connected and $H_{2}$ has at most two components. Furthermore, $H_{1}$ contains $e_{1},e_{2}$, thus $m_{1}\geq 2$.\\

If $H_{2}$ is connected, we have next inequalities,which is contradiction with $\tau(H)= \frac{2m+1}{3}$.

\[\tau(H)\leq \tau(H_{1})+\tau(H_{2})\leq 1+\frac{2m_{2}+1}{3}=1+\frac{2(m-m_{1})+1}{3}\leq 1+\frac{2(m-2)+1}{3}=\frac{2m}{3}<\frac{2m+1}{3}\]

If $H_{2}$ has two components, denoted as $H_{3}$ and $H_{4}$, and $H_{3}$ contains the cycle $C$. Because $H(V,E)$ is a counterexample with minimum edges, we have next inequalities, also a contradiction with the assumption $\tau(H)= \frac{2m+1}{3}$.

\[\tau(H_{3})\leq\frac{2m_{3}}{3},\tau(H_{4})\leq\frac{2m_{4}+1}{3}\Rightarrow\]
\[\tau(H)\leq \tau(H_{1})+\tau(H_{3})+\tau(H_{4})\leq 1+\frac{2m_{3}}{3}+\frac{2m_{4}+1}{3}=1+\frac{2(m-m_{1})+1}{3}\leq 1+\frac{2(m-2)+1}{3}=\frac{2m}{3}\]

Above all, in whatever case, there always exists a contradiction. Thus our assumption that there exists a hypertree $T$ with $|E(T)|\geq 2$ doesn't hold on and
every hypertree induced by $E(H)\setminus E(C)$ must be an edge.\\

Finally, let us consider the set of single edges induced by $E(H)\setminus E(C)$.\\

Case $1$: there exists a single edge $e$ connected with $C$ by a non-join vertex. Then we have a partial structure in Figure~\ref{graphs-5}. Now we can take $\{e,e'\}$ as $E_{1}$ and other edges as $E_{2}$. It is easy to know the edge-induced subhypergraphs $H_{1}$ and $H_{2}$ are both connected. Thus we have next inequalities,which is contradiction with $\tau(H)= \frac{2m+1}{3}$.

\[\tau(H)\leq \tau(H_{1})+\tau(H_{2})\leq 1+\frac{2(m-2)+1}{3}=\frac{2m}{3}<\frac{2m+1}{3}\]

\begin{figure}[h]
\begin{minipage}[h]{0.5\linewidth}
\centering
\includegraphics[width=2in]{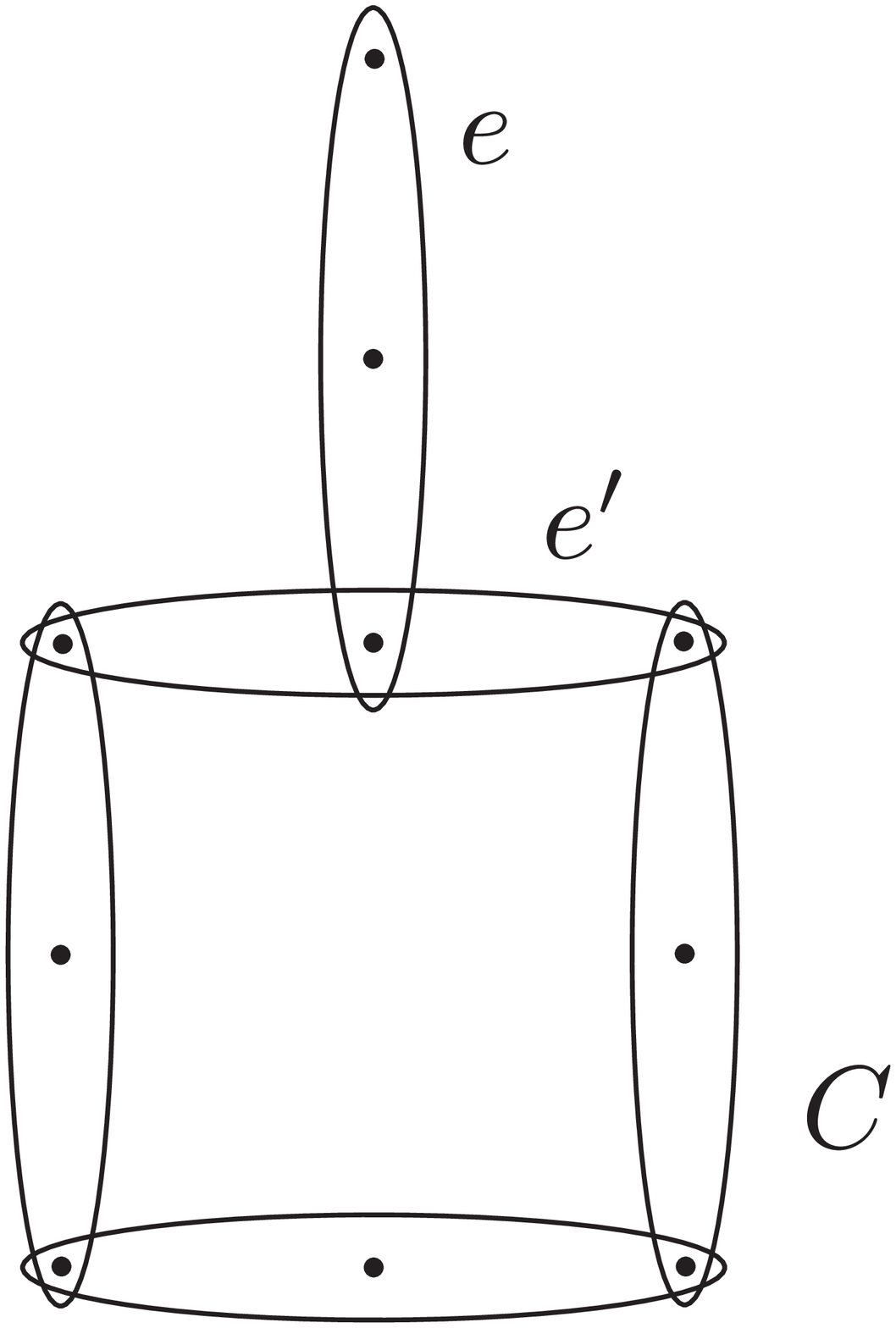}
\caption{}
\label{graphs-5}
\end{minipage}%
\begin{minipage}[h]{0.5\linewidth}
\centering
\includegraphics[width=2in]{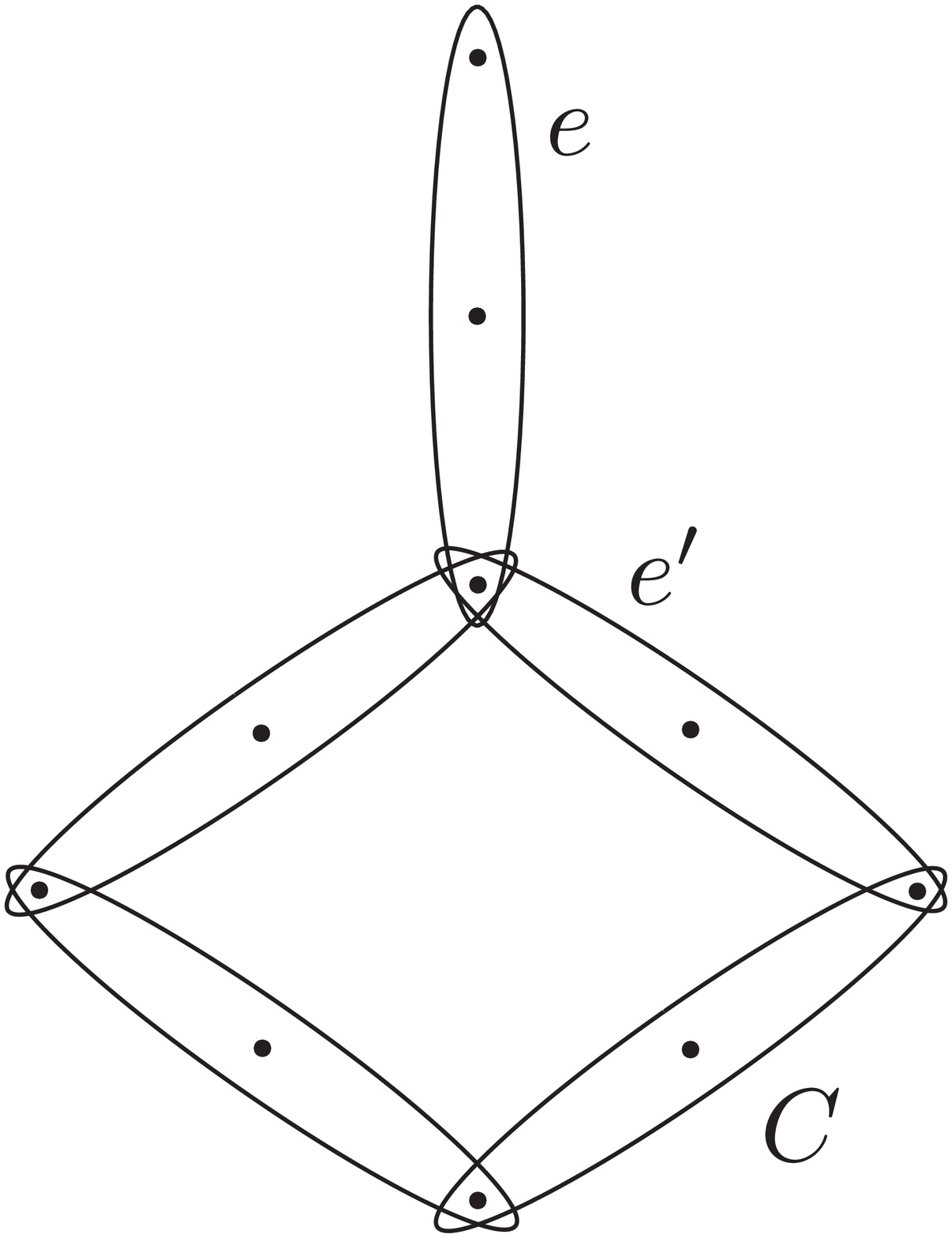}
\caption{}
\label{graphs-6}
\end{minipage}
\end{figure}

Case $2$: Every single edge $e$ is connected with $C$ by join vertices. This means every non-join vertex is not connected with the set of single edges induced by $E(H)\setminus E(C)$. Then we have a partial structure in Figure~\ref{graphs-6}. Now we can take $\{e,e'\}$ as $E_{1}$ and other edges as $E_{2}$. Because every non-join vertex is not connected with the set of single edges induced by $E(H)\setminus E(C)$. It is easy to know the edge-induced subhypergraphs $H_{1}$ and $H_{2}$ are both connected. Thus we have next inequalities,which is contradiction with $\tau(H)= \frac{2m+1}{3}$.

\[\tau(H)\leq \tau(H_{1})+\tau(H_{2})\leq 1+\frac{2(m-2)+1}{3}=\frac{2m}{3}<\frac{2m+1}{3}\]

Above all, in whatever case, there always exists a contradiction. Thus our initial assumption that $H(V,E)$ contains cycles doesn't hold on. Thus $H(V,E)$ is a hypertree with perfect matching.

\end{proof}

\section{Conclusion and future work}

In this paper, we prove that for every $3-$uniform connected hypergraph $H(V,E)$, $\tau(H)\leq \frac{2m+1}{3}$ holds on where $\tau(H)$ is the vertex cover number and $m$ is the number of edges. Furthermore, the equality holds on if and only if $H(V,E)$ is a hypertree with perfect matching. We also prove some lemmas about $3-$uniform hypergraph. These lemmas may be useful in solving some other problems. \\

In future, we can consider the vertex cover number of $k-$uniform hypergraph where $k\geq 4$. As $k$ is larger, structure analysis is more difficult because the possibilities are more and more. To solve this problem, it may be a good way to start with $k-$uniform and linear hypergraph. The linear restriction on $k-$uniform hypergraph can decrease many possibilities. If we can get some nontrivial bounds of the vertex cover number for $k-$uniform and linear hypergraph, we can generalize the bounds to all $k-$uniform hypergraphs.

\bibliography{ref}

\end{document}